\documentclass[10pt,draft,reqno]{amsart}

\usepackage{mathrsfs,amsthm,amsmath,amssymb,amstext,latexsym,amsfonts,graphicx}
\usepackage[latin1]{inputenc}
\usepackage[T1]{fontenc}
\usepackage[dvips]{epsfig}
\usepackage{pdfsync}

\def\F#1{{\mathfrak{F}_{#1}}}

\def\esp#1#2{\mathbb{E}_{#1}\left[ #2 \right]}

\def\R{{\mathbb R}}
\def\N{{\mathbb N}}
\def\car{{\mathbf 1}}
\def\d{\text{ d}}
\def\T{{\mathcal T}}
\def\L{{\mathcal L}}
\def\ee{\varepsilon}

\def\M{{\mathfrak M}}
\def\P{{\text{P}\!}}

\newcommand{\Lip}{\operatorname{Lip}}
\newcommand{\Dom}{\operatorname{Dom}}
\newcommand{\id}{\operatorname{Id}}
\newcommand{\esssup}{{\operatorname{esssup}}}
\def\GD{\nabla^\sharp}
\def\GC{\nabla^c}

\def\LD{{\mathcal L}^\sharp}
\def\LC{{\mathcal L}^c}

\def\PD{P^\sharp}
\def\PC{P^c}

\def\QC{Q^c}
\def\QD{Q^\sharp}

\def\LC{{\mathcal L}^c}
\def\LD{{\mathcal L}^\sharp}

\def\alphaC{\alpha^c}
\def\alphaD{\alpha^\sharp}

\def\DD{\delta^\sharp}
\def\DC{\delta^c}
\def\B{{\mathcal B}}

\def\hX{\widehat{\Lambda}}
\def\hmu{\widehat{\mu}}
\def\hd2{\widehat{\rho _1}}
\def\hrho{\widehat{\rho _2}}
\def\homega{\widehat{\omega}}
\def\hF{\widehat{F}}
\def\hL{\widehat{L}}
\def\hnu{\widehat{\nu}}
\def\heta{\widehat{\eta}}

     \makeatletter
     \def\section{\@startsection{section}{1}%
     \z@{.7\linespacing\@plus\linespacing}{.5\linespacing}%
     {\bfseries
     \centering
     }}
     \def\@secnumfont{\bfseries}
     \makeatother
\newtheorem{theorem}{Theorem}[section]
\newtheorem{lemma}[theorem]{Lemma}
\newtheorem{proposition}[theorem]{Proposition}

\newtheorem{hyp}[theorem]{Hypothesis}
\theoremstyle{definition}
\newtheorem{definition}[theorem]{Definition}

\theoremstyle{remark}

\numberwithin{equation}{section}
\begin{document}

\title[Rubinstein distances on configuration spaces]{Upper bounds on Rubinstein distances on configuration spaces and applications}

\author[Laurent Decreusefond]{Laurent Decreusefond}
\address{Laurent Decreusefond: Institut TELECOM, TELECOM ParisTech, CNRS LTCI, Paris, France}
\email{Laurent.Decreusefond@telecom-paristech.fr}

\author[Ald\'eric Joulin]{Ald\'eric Joulin}
\address{Ald\'eric Joulin: Universit\'e de Toulouse, Institut National des Sciences Appliqu\'ees, Institut de Math\'ematiques de Toulouse, F-31077 Toulouse, France}
\email{alderic.joulin@math.univ-toulouse.fr}

\author[Nicolas Savy]{Nicolas Savy}
\address{Nicolas Savy: Universit\'e de Toulouse, Universit\'e Paul Sabatier, Institut de Math\'ematiques de Toulouse, F-31062 Toulouse, France}
\email{nicolas.savy@math.univ-toulouse.fr}

\subjclass[2000] {60G55,60H07,60E15}

\keywords{Configuration space, Poisson measure, Rubinstein distance, Malliavin derivative, Rademacher property, tail estimate, isoperimetry}

\begin{abstract}
In this paper, we provide upper bounds on several Rubinstein-type
  distances on the configuration space equipped with the Poisson
  measure. Our inequalities involve the two well-known gradients, in
  the sense of Malliavin calculus, which can be defined on this space. Actually, we show that depending on the distance between
  configurations which is considered, it is one gradient or the other
  which is the most effective. Some applications to distance estimates
  between Poisson and other more sophisticated processes are also provided,
  and an application of our results to tail and isoperimetric estimates completes this work.
\end{abstract}

\maketitle

\section{Introduction}
\label{sec:introduction}
\paragraph{}
Let $\Lambda$ be a $\sigma$-compact metric space and $\Gamma_\Lambda$ be the space of
configurations on $\Lambda$ equipped with a Poisson measure
$\mu$. Defining and evaluating some distances between probability
measures on $\Gamma_\Lambda$ is an important problem, both theoretical
and for applications, since it is equivalent to defining distances
between point processes (see for instance Chapters~2 and 3 of
\cite{Schuhmacher:2005lx} for a thorough discussion and references
about this topic). Among the large class of distances one may consider, the one we want to study relies on an optimal transportation problem. Letting $\rho$ be a lower semi-continuous distance on $\Gamma_\Lambda$ and two configurations $\omega, \eta \in \Gamma_\Lambda$, we understand the quantity $\rho (\omega , \eta)$ as the cost for transporting one unit of mass from $\omega$ to $\eta$. Hence the optimal transportation cost between $\mu$ and some probability measure $\nu$ on $\Gamma_\Lambda$ is given by
\begin{equation*}
\T_\rho (\mu, \nu) = \inf_{\gamma \in \Sigma(\mu, \nu)} \int_{\Gamma_\Lambda} \int_{\Gamma_\Lambda} \rho(\omega ,\eta ) \d \gamma(\omega ,\eta),
\end{equation*}
where $\Sigma(\mu, \nu)$ is the set of probability measures on $\Gamma_\Lambda \times \Gamma_\Lambda$ with marginals $\mu$ and $\nu$.~Such a quantity is called the Rubinstein distance between $\mu$ et $\nu$.~Being defined by a variational formula, its explicit expression is of difficult access in general but might be estimated from above:~the construction of any coupling between $\mu$ and $\nu$ yields a bound on the Rubinstein distance between $\mu$ and $\nu$.~In particular, a convenient upper bound ensures its finiteness, which is not guaranteed a priori.
\paragraph{}Another interesting property of $\T_\rho$ is its rich duality.~More precisely,~the Kanto\-ro\-vich-Rubinstein duality allows us to rewrite the Rubinstein distance as
\begin{equation*}\T_\rho (\mu , \nu) = \sup_{F\in \rho -\Lip_1}\int _{\Gamma_\Lambda} F\d (\mu -\nu ),\end{equation*}
where $\rho -\Lip_1$ denotes the set of $1$-Lipschitz functions on $\Gamma_\Lambda$ with respect to the distance $\rho$. This means that $\T_\rho$ depends crucially on the distance on the configuration space as it changes the set of Lipschitz functions, hence incorporates a lot of information on the geometry of $\Gamma _\Lambda$. Using the dual definition of the Rubinstein distance instead of the original one can be very relevant in some cases.
\paragraph{}
Given a probability measure $\nu$ with density $L$ with respect to the Poisson refe\-rence measure $\mu$, our purpose in the present paper is to control from above the Rubinstein distance $\T_\rho (\mu, \nu)$ in terms of convenient (and easily computable) quantities involving the density $L$. Such inequalities belong to the domain of functional inequalities, which is by now a wide field of research with numerous methods of proofs. See for instance the very complete monograph \cite{Villani:2007fk} and particularly Chapters~21 and 22 for a large panorama on this topic, with precise references and credit.
\paragraph{} 
To derive our inequalities, the two main ingredients at work are other representations of the Rubinstein distance and the Rademacher property. On the one hand, such representations can be obtained either by embedding the two probabi- lity measures into the evolution of a Markov semi-group, or by using the so-called Clark formula. On the other hand, the Rademacher property formally states that given a distance $\rho$, there exists a notion of gradient such that its domain contains the set $\rho -\Lip_1$ and any function in $\rho -\Lip_1$ has a gradient whose norm is less than $1$, i.e., that we can proceed as in finite dimension.
\paragraph{}
For these two steps, we need a notion of gradient. In the setting of configuration spaces, such a notion does exist within the Malliavin
calculus. In fact, we even have two notions of gradient: a ``differential'' gradient (see \cite{MR99d:58179,
  MR1730565}) and a gradient expressed as a finite difference operator (see \cite{nualart88_1}). We show that depending on the
distance $\rho$ chosen on the configuration space, one gradient or the other is more convenient, i.e., the Rademacher property holds with one
notion of gradient, or the other. 
\paragraph{}
The paper is organized as follows. After the preliminaries of Section~\ref{sec:preliminaries}, we provide in
Section~\ref{sec:bounds} various upper bounds on the Rubinstein distance $\T_\rho (\mu, \nu)$, where $\rho$ is the total variation distance, the Wasserstein distance or the trivial distance on the configuration space $\Gamma _\Lambda$. Based on a semi-group approach, the first abstract upper bound involves the gradient associated to our given distance $\rho$ in the sense of the Rademacher property. When dealing with the total variation distance on the one hand, such an estimate has a simplified expression, contained in our first main result, Theorem~\ref{T:E1}, which can be retrieved by using an alternative method, namely the Clark formula. On the other hand, when the configuration space is equipped with the Wasserstein distance, the upper bound we give in our second main result, Theorem~\ref{thm:time_change}, relies on a time-change argument together with the Girsanov Theorem. Finally, the last Section~\ref{sec:applications} is devoted to numerous applications of these two inequalities:~by choosing the probability measure $\nu$ as the distribution of a given process, we are able to estimate from above distances between Poisson processes, between Poisson and Cox processes, between Poisson and Gibbs processes, etc.~We thus hope to give a systematic treatment of the various situations one may encounter in applications.~We conclude this work by providing another consequence of Theorem~\ref{T:E1} to tail and isoperimetric estimates.~In particular, we obtain sharp deviation inequalities for the total variation distance and also a new estimate of the classical isoperimetric constant, which is asymptotically sharp as the total mass of $\Lambda$ is small.
\section{Preliminaries}
\label{sec:preliminaries}
\paragraph{}
Let $X$ be a Polish space and $\rho$ a lower semi-continuous distance
on $X\times X$, which does not necessarily generate the topology on
$X$. Given two probability measures $\mu$ and $\nu$ on $X$, the
optimal transportation problem associated to $\rho$ consists in
evaluating the distance
\begin{equation}\label{eq:19}
  \T_\rho (\mu,\nu)= \inf_{\gamma \in \Sigma(\mu, \nu)} \int_X \int_X   \rho(x,y) \d \gamma(x,y),
\end{equation}
where $\Sigma(\mu, \nu)$ is the set of probability measures on
$X\times X$ with first (respectively second) marginal $\mu$
(respectively $\nu$). By Theorem~4.1 in \cite{Villani:2007fk}, there
exists at least one probability measure $\gamma$ for which the infimum
is attained. According to the celebrated Kantorovitch-Rubinstein
duality theorem, cf. Theorem~5.10 in \cite{Villani:2007fk}, this minimum is equal to
\begin{equation}
\label{eq:9} \T_\rho (\mu, \nu) = \sup_{\substack{F \in \rho-\Lip_1 \\ F\in L^1(\mu+\nu)}}\int _X F \d (\mu-\nu) ,
\end{equation}
where $\rho-\Lip_m$ is the set of bounded Lipschitz continuous functions $F$ from $X$ to $\R$ with Lipschitz constant $m$:
\begin{equation*}
  |F(x)-F(y)| \leq m \rho(x,y), \quad x,y \in X .
\end{equation*}
In the context of optimal transportation, $\T_\rho$ is considered as a Rubinstein distance since the cost function is already a distance (see
for instance the bibliographical notes at the end of Chapter~6 in \cite{Villani:2007fk}).
\paragraph{}In this paper, we consider the situation where $X=\Gamma_{\Lambda}$ is the configuration space on a $\sigma$-compact metric space $\Lambda$ with Borel $\sigma$-algebra ${\mathcal B}(\Lambda )$, i.e.,
\begin{equation*}
  \Gamma_\Lambda = \{ \omega \subset \Lambda : \omega \cap K \text{ is a finite set for every compact } K\in {\mathcal B}(\Lambda ) \}.
\end{equation*}
Here the $\sigma$-compactness means that $\Lambda$ can be partitioned into the union of coun\-tably many compact subspaces.
We identify $\omega \in \Gamma_\Lambda$ and the positive Radon measure $\sum_{x\in \omega } \varepsilon_x,$ where $\ee_a$ is the Dirac
measure at point $a.$ Throughout this paper, $\Gamma_\Lambda$ is endowed with the vague topology, i.e., the weakest topology such that
for all $f\in {\mathcal C}_0 (\Lambda )$ (continuous with compact support on $\Lambda$), the following maps
\begin{equation*}
  \omega \mapsto \int_{\Lambda} f\d \omega = \sum_{x\in\omega } f(x)
\end{equation*}
are continuous. When $f$ is the indicator function of a subset $B,$ we will use the shorter notation $\omega (B)$ for the integral of
$\car_B$ with respect to $\omega $.~We denote by $\mathcal{B}(\Gamma_{\Lambda})$ the corresponding Borel $\sigma$-algebra.~Let $\M(\Lambda)$ be the space of positive and diffuse Radon measures on $\mathcal{B}(\Lambda)$ endowed with the corresponding Borel $\sigma$-field and equipped with the topology of vague con\-ver\-gen\-ce.~Given a measure $\sigma \in \M(\Lambda)$, the probability space under consideration in the remainder of this paper will be the Poisson space $(\Gamma_\Lambda,\mathcal{B}(\Gamma_{\Lambda}),\mu_\sigma)$, where $\mu_\sigma$ is the Poisson measure of intensity $\sigma$, i.e., the probability measure on $\Gamma_\Lambda$ fully characterized by
\begin{equation*}
\esp{\mu_\sigma}{\exp \left( \int_{\Lambda} f \d \omega \right) } = \exp \left \{ \int_{\Lambda} ( e^f -1) \d \sigma \right \} ,
\end{equation*}
for all $f\in {\mathcal C}_0 (\Lambda )$. Here $\mathbb{E}_{\mu_\sigma }$ stands for the expectation under the measure $\mu _\sigma $.
\subsection{Distances on the configuration space $\Gamma _\Lambda $}
\label{sec:distances}
Actually, several distance concepts are available between elements of the configuration space $\Gamma_\Lambda$, cf.~for instance \cite{Schuhmacher:2005lx} for a thorough discussion about this topic.~We introduce only three of them which will be useful in the sequel.~Let $\omega$ and $\eta$ be two configurations in $\Gamma_\Lambda$.
\begin{description}
\item[Trivial distance] The trivial distance is simply given by
  \begin{eqnarray*}
    \rho _0 (\omega ,\eta ) & = & \car_{\{\omega \neq \eta \}}.
  \end{eqnarray*}
\item[Total variation distance] The total variation distance is defined as
  \begin{eqnarray*}
    \rho _1(\omega , \eta ) &= & \sum _{x \in \Lambda } \vert \omega (\{x \}) - \eta (\{x \}) \vert \\
    & = & \omega\Delta\eta(\Lambda)+\eta\Delta\omega(\Lambda),
  \end{eqnarray*}
  where $\omega\Delta\eta=\omega\backslash (\omega\cap \eta).$ 
\item[Wasserstein distance] If $\Lambda = \R^k$ and $\kappa $ is the
  Euclidean distance, the Wasserstein distance is given by
  \begin{equation*}\label{eq:7}
    \rho _2 (\omega , \eta )= \inf _{\beta\in \Sigma({\omega , \eta })} \, \sqrt{\int _\Lambda \int _\Lambda \kappa (x,y) ^2 \d \beta (x,y) },
  \end{equation*}
  where $\Sigma({\omega , \eta })$ denotes the set of configurations $\beta \in \Gamma_{\Lambda\times \Lambda}$ having marginals $\omega
  $ and $\eta $, see \cite{Decreusefond:2006cv,MR1730565}.
\end{description}
\paragraph{}Let us comment on these notions of distance on the configuration space $\Gamma_\Lambda$. First, the total variation distance $\rho _1$ is nothing but the number of different atoms between two configurations. In particular, we allow them to be infinite so that the total variation distance might take infinite values. Note that our definition is a straightforward generalization of the classical notion of total variation distance between probability measures, since it coincides with the usual definition when the configurations are normalized by their total masses.
\paragraph{}As the total variation distance $\rho _1$, the Wasserstein distance $\rho _2 $ also shares the property that it might takes infinite
values. Indeed, if the total masses of two configurations $\omega $ and $\eta $ are  finite but differ, then there exists no coupling configuration $\beta$ in $\Sigma({\omega ,\eta })$, hence the distance should be infinite. If $\omega (\Lambda)=\eta (\Lambda)< + \infty$ with $\omega = \sum _{j=1} ^{\omega (\Lambda)} \delta _{x_j}$ and $\eta = \sum _{j=1} ^{\eta (\Lambda)} \delta _{y_j}$, we can also write
\begin{equation*}
  \rho _2 (\omega ,\eta ) ^2 = \inf_{\tau \in {\mathfrak S}_{\omega (\Lambda)}}\sum_{j=1}^{\omega (\Lambda)} \kappa (x_j, y_{\tau (j)})^2 ,
\end{equation*}
where ${\mathfrak S}_{\omega (\Lambda)}$ denotes the symmetric group
on the finite set $\{ 1,2, \ldots , \omega (\Lambda) \}$.  As such
$\rho _2$ appears as the dimension-free generalization of the Euclidean distance.
\paragraph{}In order to use the Kantorovich-Rubinstein duality Theorem, the
lower semi-continuity of the distances $\rho _i$, $i\in \{ 0,1,2\}$, is required. This is
the object of the next lemma.
\begin{lemma}
  \label{lem:lsc}
  For any $i\in \{ 0,1,2\}$, the distance $\rho_i$ is lower semi-continuous on the product
  space $\Gamma_\Lambda\times \Gamma_\Lambda$ equipped with the product topology.
\end{lemma}
\begin{proof}
  It is immediate for the trivial distance $\rho _0$ and it is proved
  in Lemma~4.1 in \cite{MR1730565} for the Wasserstein distance
  $\rho_2$. To verify this property for the total variation distance
  $\rho_1$, let $\alpha$ be a real number and consider $J_\alpha$
  defined by
  \begin{equation*}
    J_\alpha = \{(\omega, \eta) \in \Gamma_\Lambda\times \Gamma_\Lambda :\ \rho_1(\omega, \eta)\le \alpha\}.
  \end{equation*}
  Let $((\omega_n, \eta_n) , \, n\ge 1)$ converge vaguely to $(\omega,
  \eta)$ and such that for any $n$, $(\omega_n, \eta_n)$ belongs to
  $J_\alpha.$ By the triangular inequality, we have for any compact
  set $K$ and any $n$:
  \begin{eqnarray*}
    \rho _1(\pi_K\omega, \pi_K\eta) &\leq & \rho _1(\pi_K\omega, \pi_K\omega _n) + \alpha + \rho _1(\pi_K\eta _n, \pi_K\eta) ,
  \end{eqnarray*}
  where $\pi_K$ denotes the restriction to $K$ of a
  configuration. Hence using the vague convergence, we obtain that
  $(\pi_K\omega , \pi_K\eta) \in J_\alpha $. Finally, since the metric space $\Lambda$ is $\sigma$-compact, the monotone
  convergence theorem for an exhaustive sequence of compacts
  $(K_p)_{p\in \N}$ entails that
  \begin{equation*}
    \rho_1(\omega, \eta) = \lim _{p\to +\infty} \rho_1(\pi _{K_p} \omega, \pi _{K_p}\eta) \le \alpha ,
  \end{equation*}
  hence the set $J_\alpha$ is vaguely closed.
\end{proof}
\paragraph{}Let us mention that Lemma~\ref{lem:lsc} entails the lower semi-continuity of the Rubinstein distance $\T _{\rho _i}$, $i\in \{ 0,1,2\}$, with respect to the weak topology on the space of probability measures on $\Gamma _\Lambda$, cf. for instance Remark~6.12 in \cite{Villani:2007fk}. In particular, since the space $\M(\Lambda)$ is equipped with the vague topology, then the application $\sigma \mapsto \mu _\sigma$ is continuous so that the mapping $\sigma \mapsto \T _{\rho _i} (\mu _{\sigma } , \nu)$, $i\in \{ 0,1,2\}$, is lower semi-continuous for any given probability measure $\nu$ on $\Gamma _\Lambda$. However for $i \in \{ 1, 2 \}$, the Rubinstein distances $\T _{\rho _i}$ is not continuous and might be infinite since the distance $\rho _i$ is very often infinite itself, as in the Wiener space situation of \cite{MR2036490}.
\paragraph{}Actually, we mention that our definitions do not coincide with some of
the usual definitions of (bounded) distances between point processes,
see for instance \cite{MR1708412,MR93g:60043,Schuhmacher:2005lx}. As
mentioned above, it is customary to use the classical notion of total
variation by considering normalized configurations, i.e.,
\begin{equation*}
  \widetilde{\rho _1}(\omega , \eta )=\rho _1 \left( \frac{\omega }{\omega (\Lambda )}, \frac{\eta}{\eta (\Lambda )} \right) ,
\end{equation*}
provided both configurations have finite total masses. It should be
noted that since $ \widetilde{\rho _1}$ is not lower semi-continuous,
the Kantorovich-Rubinstein duality Theorem is no longer satisfied, so
that we cannot use the identity (\ref{eq:9}) in our framework. For
instance, let $\Lambda=\R$, $\omega=\ee_0$ and
$\eta=\ee_1$. Choose $\omega_n=\ee_0+\ee_n$ and $\eta_n=\ee_1+\ee_n$. As $n$ goes to infinity, $\omega_n$ and $\eta_n$ tend vaguely to $\omega$ and $\eta$ respectively. However, we have $\widetilde{\rho _1}(\omega , \eta )=2$ whereas $\widetilde{\rho _1}(\omega_n , \eta_n )= 1,$ for any integer $n \geq 2$. \\
It is also customary to replace $\rho_2$ by $ \widetilde{\rho _2}$
defined by
\begin{equation*}
  \widetilde{\rho _2}(\omega ,\eta )=
  \begin{cases}
    \frac{1}{\omega(\Lambda)} \, \rho _2 (\omega ,\eta )&\text{ if } \omega(\Lambda)=\eta(\Lambda)\neq 0,\\
    |\omega (\Lambda)-\eta (\Lambda)| & \text{ otherwise.}
  \end{cases}
\end{equation*}
The normalization by the inverse of $\omega(\Lambda)$ shrinks the
$\rho_2$ distance by a factor roughly equal to the expectation of
$\omega(\Lambda)^{-1}$, see \cite{Decreusefond:2006cv}. More
importantly, the term $ |\omega (\Lambda)-\eta (\Lambda)|$ has no
dimension (in the sense of dimensional analysis) whereas the term
involving $\rho_2$ has the dimension of a length.
Furthermore, the distance $\rho_2$ has interesting geometric properties of the space $\Gamma_\Lambda$ like the Rademacher property (see Lemma~\ref{lem:lip_wasserstein} below), not shared by $ \widetilde{\rho _2}.$
\subsection{Malliavin derivatives and the Rademacher property}
Before introducing the so-called Rademacher property on the
configuration space $\Gamma_\Lambda$, we need some additional
structure.
\begin{hyp} \label{H:1} Assume now that we have:
  \begin{itemize}
  \item A kernel $Q$ on $\Gamma_\Lambda\times \Lambda$, i.e. $Q(\cdot
    ,A)$ is measurable as a function on $\Gamma_{\Lambda}$ for any
    $A\in {\mathcal B}(\Lambda)$ and $Q(\omega , \cdot )$ is a positive Radon measure on $\B(\Lambda)$ for any
    $\omega\in\Gamma_\Lambda.$ We set $\d  \alpha(\omega, x) =
    Q(\omega, \d  x) \, \d  \mu _{\sigma }(\omega).$
  \item A gradient/Malliavin derivative $\nabla$, defined on a dense
    subset $\Dom \nabla$ of $L^2 (\mu _{\sigma })$, such that for any
    $F\in \Dom \nabla$,
$$
\int _{\Gamma _\Lambda } \int_{\Lambda} \vert \nabla_x F (\omega )
\vert ^2 \d \alpha  (\omega ,x) < + \infty ,
$$
\end{itemize}
i.e., the domain of the gradient is $\Dom \nabla = \{ F \in L^2 (\mu _{\sigma }) : \nabla F \in L^2 (\alpha )\}$.
\end{hyp}
\paragraph{} We say that a process $u=u(\omega, x)$ belongs to $\Dom \delta$
whenever there exists a constant $c$ such that for any $F\in \Dom
\nabla$,
\begin{equation*}
  \left \vert \int _{\Gamma _\Lambda } \int_{\Lambda} \nabla_{x} F (\omega ) \, u( \omega ,x ) \d \alpha (\omega , x) \right \vert \le c \Vert F\Vert _{L^2 (\mu _{\sigma } )}.
\end{equation*}
For such a process, we define the operator $\delta$ by duality:
\begin{equation} \label{E:DEFQ} \int_{\Gamma _\Lambda } \int_{\Lambda}
  \nabla_x F(\omega) \, u(\omega, x) \d \alpha (\omega , x) = \int
  _{\Gamma _\Lambda } F (\omega ) \, \delta u (\omega ) \d \mu
  _{\sigma } (\omega ) .
\end{equation}
Denote the self-adjoint operator $\L = \delta \nabla$ acting on its domain $\Dom \L \subset \Dom \nabla$ and let $(P_t)_{t\geq 0}$ be the associated
Ornstein-Uhlenbeck semi-group, i.e. the semi-group whose infinitesimal generator is $-\L$.
\paragraph{}Once the stochastic gradient has been introduced, let us relate it to
the geometry of the configuration space $\Gamma _\Lambda $.
\begin{definition}
  \label{def:rademacher}
  Given a distance $\rho$ and a gradient $\nabla$ on $\Gamma
  _\Lambda$, we say that the couple $(\nabla, \rho )$ has the
  Rademacher property whenever
\begin{equation}
\label{eq:rademacher}
\rho -\Lip_1 \subset \Dom \nabla \quad \quad and \quad \quad \vert \nabla_x F (\omega) \vert \le 1, \quad \alpha \text{-a.e.}
\end{equation}
\end{definition}
\paragraph{}To investigate the Rubinstein distance associated to a distance on
$\Gamma_\Lambda$, it will be of crucial importance to find the
convenient notion of gradient for which the Rademacher property
holds.
\paragraph{} {\bf \textit{Discrete gradient on configuration space.}} 
Given a functional $F\in L^2 (\mu _\sigma )$, the discrete gradient of $F$,
denoted by $\GD F$, is defined by
\begin{equation*}
  \GD_x F(\omega )= F(\omega + \varepsilon_x)-F(\omega ) , \quad (\omega , x) \in \Gamma _\Lambda \times \Lambda .
\end{equation*}
In particular, $\Dom \GD$ is the subspace of $L^2 (\mu _\sigma )$
random variables such that
\begin{equation*}
  \esp{\mu _\sigma }{\int_\Lambda |\GD_x F|^2 \d\sigma(x)}<+\infty.
\end{equation*}
We set $\QD(\omega , \d x)= \d \sigma (x)$ so that $\alphaD=\mu
_{\sigma }\otimes\sigma.$
The $n$-th multiple stochastic integral of a real-valued
square-integrable symmetric function $f_n \in L^2 (\sigma ^{\otimes
  n})$ is defined as
\begin{equation*}
  J_n(f_n) = \int _{\Delta _n} f_n (x_1,\ldots,x_n)\, \d(\omega -\sigma)(x_1)\ldots\d(\omega -\sigma)(x_n) ,
\end{equation*}
where $\Delta _n = \{ (x_1, \ldots , x_n) \in \Lambda ^n ,\ x_i \neq
x_j , \, i\neq j \}$. As a convention, we identify $L^2 (\sigma ^{\otimes
  0})$ to $\R$ and let $J_0(f_0) = f_0$, $f_0 \in L^2 (\sigma ^{\otimes
  0}) \simeq \R$. We have the isometry formula
\begin{equation}
  \label{eq:isometrie}
  \esp{\mu _{\sigma } }{J_n (f_n)J_m (f_m)} = n! \, \car _{ \{ n=m\} } \, \int _{\Lambda ^n} f_n \, f_m \, \d \sigma ^{\otimes n}.
\end{equation}
According to \cite{ruiz85,nualart88_1}, the Chaotic Representation
Property holds on the configuration space, i.e., every functional
$F\in L^2 (\mu _{\sigma } )$ can be written as
\begin{equation*}
  F= \esp{\mu _{\sigma } }{F} + \sum_{n=1}^{+\infty} J_n(f_n) .
\end{equation*}
Moreover, if $F \in \Dom \GD$, then the discrete gradient acts on multiple stochastic integrals as
\begin{align*}
  \GD_x F = \sum_{n=1}^{+\infty} n J_{n-1}(f_n (\cdot ,x)) , \quad \alphaD \text{-a.e.}
\end{align*}
Denote $\DD$ the adjoint operator of $\GD$ in the sense of
(\ref{E:DEFQ}). Then the self-adjoint number operator $\LD = \DD \GD$ has the following
expression in terms of chaos:
$$
\LD F = \sum _{n=1} ^{+\infty } n J_n (f_n) ,
$$
whenever $F \in \Dom \LD$, and the associated Ornstein-Uhlenbeck semi-group $( \PD_t)_{t\geq 0}$ is given by
$$
\PD _t F = \esp{\mu _{\sigma } }{F} + \sum _{n=1} ^{+\infty } e^{-nt} J_n (f_n).
$$
Hence the invariance property of the Poisson measure $\mu _{\sigma }$ with respect to the semi-group reads as $\mathbb{E} _{\mu _{\sigma }} [\PD _t F ] = \esp{\mu _{\sigma } }{F}$. Moreover, we have the commutation property between gradient and semi-group, which will be useful in the sequel: if $F \in \Dom \GD$,
\begin{equation}
  \label{eqn:commutation}
  \GD_{x} \PD_t F = e^{-t} \PD_t \GD_x F , \quad x \in
  \Lambda, \quad \, t \geq 0 .
\end{equation}
By the isometry formula~(\ref{eq:isometrie}), the semi-group is exponentially ergodic in $L^2(\mu _\sigma)$ with respect to the Poisson measure $\mu _{\sigma }$, i.e., for any $t\geq 0$,
\begin{eqnarray*}
\Vert P_t F - \esp{\mu _{\sigma } }{F} \Vert _{L^2 (\mu _\sigma)} ^2 & = & \sum _{n\geq 1} e^{-2nt} \, \esp{\mu _{\sigma } }{J_n (f_n) ^2} \\
& \leq & e^{-2t} \, \Vert F - \esp{\mu _{\sigma } }{F} \Vert _{L^2 (\mu _\sigma)} ^2 .
\end{eqnarray*}
\paragraph{}Using the discrete gradient, the distances of interest on
$\Gamma_\Lambda$ are the trivial distance $\rho _0$ and the total
variation distance $\rho _1$, as illustrated by the following Lemma.
\begin{lemma}\label{lem:lip_discrete} Assume that the intensity measure
  $\sigma$ is finite on $\Lambda$. Then the couples $(\GD, \rho _0)$ and
  $(\GD, \rho _1)$ satisfy the Rademacher property $(\ref{eq:rademacher})$.
\end{lemma}
\begin{proof}
  Letting $F \in \rho_i -\Lip_1$, $i \in \{ 0,1 \},$ we have by the very
  definition of the discrete gradient:
$$
\vert \GD_x F(\omega ) \vert = \vert F(\omega +\varepsilon_x)-F(\omega
)\vert \leq \rho _i (\omega +\varepsilon_x, \omega ) \leq 1.
$$
Since $\sigma$ is finite, it follows that
\begin{equation*}
  \int _\Lambda \vert \GD_x F(\omega ) \vert ^2 \d\sigma (x) \le \sigma(\Lambda),
\end{equation*}
hence that $F$ belongs to $\Dom \GD$. The proof is achieved.
\end{proof}
\paragraph{}Note that the converse direction holds for the total variation
distance $\rho _1$. Indeed, consider two configurations $\omega$ and
$\eta$. If $\rho _1(\omega, \eta)=+\infty$, there is nothing to
prove. If $\rho _1(\omega, \eta)$ is finite, then since $|\GD_x
F(\omega )|\le 1$, $\alpha ^\sharp$-a.e., we get
\begin{align*}
  |F(\eta)-F(\omega)|&\le |F(\eta \cap \omega \, \cup \, \eta\Delta
  \omega)-F(\eta\cap \omega)| + |F(\eta \cap \omega \, \cup \,
  \omega\Delta
  \eta)-F(\eta\cap \omega)|\\
  &\le (\eta\Delta
  \omega) (\Lambda) + (\omega\Delta \eta)(\Lambda) \\
  & = \rho _1(\eta ,\omega) .
\end{align*}
\paragraph{} {\bf \textit{Differential gradient on configuration space.}} Let us introduce another stochastic gradient on the configuration space $\Gamma_\Lambda$ which is a derivation, see \cite{MR99d:58179, MR1730565}. Given the Euclidean space $\Lambda = \R ^k$, let $V(\Lambda)$
be the space of ${\mathcal C}^\infty$ vector fields on $\Lambda$ and
$V_0(\Lambda)\subset V(\Lambda),$ the subspace consisting of all
vector fields with compact support. For $v\in V_0(\Lambda),$ for any
$x\in \Lambda,$ the curve
\begin{displaymath}
  t\mapsto {\mathcal V}_t^v(x)\in \Lambda
\end{displaymath}
is defined as the solution of the following Cauchy problem
\begin{equation}
  \label{eq:D2}
  \begin{cases}
    \frac{\d}{\d t}{\mathcal V}_t^v(x) & = \, \, \, v({\mathcal V}_t^v(x)),\\
    {\mathcal V}_0^v(x) & = \, \, \, x.
  \end{cases}
\end{equation}
The associated flow $({\mathcal V}_t^v, t\in \R)$ induces a curve
$({\mathcal V}_t^v)^*\omega =\omega \circ ({\mathcal V}_t^v)^{-1}$,
$t\in \R$, on $\Gamma_\Lambda$: if $\omega =\sum_{x\in \omega
}\varepsilon_x$ then $({\mathcal V}_t^v)^*\omega =\sum_{x\in \omega
}\varepsilon_{{\mathcal V}_t^v(x)}.$ We are then in position to define
a notion of differentiability on $\Gamma_\Lambda$. We take $\QC
(\omega,\d x)= \d \omega (x)=\sum_{y\in \omega } \d
\varepsilon_{y}(x)$ and $\d\alphaC(\omega, x)=\d\omega(x)\d\mu
_{\sigma }(\omega).$ A measurable function $F :\Gamma_\Lambda \to \R$
is said to be differentiable if for any $v\in V_0(\Lambda)$, the
following limit exists:
\begin{equation*}
  \lim_{t\to 0} \, \frac{ F({\mathcal V}_t^v(\omega ))-F(\omega )}{t} .
\end{equation*}
We denote $\GC_vF(\omega )$ the preceding quantity. The domain of
$\GC$ is then the set of integrable and differentiable functions such
that there exists a process $(\omega,\, x)\mapsto \GC_x F(\omega )$
which belongs to $L^2(\alphaC)$ and satisfies
\begin{equation*}
  \GC_v F (\omega) =\int_\Lambda \GC_x F (\omega) v(x) \d \omega(x).
\end{equation*}
We denote by $\DC$ the adjoint operator of $\GC$ in the sense of (\ref{E:DEFQ}). Note that the integration in the left-hand-side of the duality formula (\ref{E:DEFQ}) is made with respect to a configuration $\omega$, whereas the intensity measure $\sigma$ is involved in the case of the discrete gradient. Given the self-adjoint operator $\LC = \DC \GC$, the associated Ornstein-Uhlenbeck semi-group $(\PC _t)_{t\geq 0}$ is ergodic in $L^2(\mu _\sigma)$ with respect to the Poisson measure $\mu _\sigma$, cf. Theorem~4.3 in \cite{MR99d:58179}. However, in contrast to the case of the discrete gradient, there is no known commutation relationship between the gradient $\GC$ and the semi-group $\PC _t$.
\paragraph{} The distance we focus on in this part is the Wasserstein distance $\rho _2$. We have the following lemma.
\begin{lemma}\label{lem:lip_wasserstein}
The couple $(\GC , \rho _2)$ satisfies the Rademacher property $(\ref{eq:rademacher})$.
\end{lemma}
\begin{proof}
  The proof is straightforward. Indeed, letting $F \in \rho
  _2-\Lip_1$, we know from Theorem~1.3 in \cite{MR1730565} that $F\in
  \Dom \GC$ and that
$$
\sum _{x\in \omega } \vert \GC _x F (\omega ) \vert ^2 = \int _\Lambda
\vert \GC _x F (\omega ) \vert ^2 \d \omega (x) \leq 1 , \quad \mu
_{\sigma }\text{-a.s.}
$$
Hence we obtain $\vert \GC _x F (\omega ) \vert \leq 1, \
\alphaC$-a.e., in other words the Rademacher property $(\ref{eq:rademacher})$ is satisfied.
\end{proof}
\section{Upper bounds on Rubinstein distances}
\label{sec:bounds}
\subsection{An abstract upper bound on Rubinstein distances}
\label{sec:abstract}
Let us establish first an abstract upper bound on the Rubinstein
distance by using a semi-group method, provided the associated couple
gradient/distance satisfies the Rademacher property (\ref{eq:rademacher}). Denote $\rho$ a lower semi-continuous distance on the configuration space $\Gamma _\Lambda$ and assume that Hypothesis~\ref{H:1} is fulfilled.
\begin{proposition}
  \label{T:E0}
Assume that the couple $(\nabla , \rho)$ satisfies the Rademacher property (\ref{eq:rademacher}). Let $L$ be the density of an absolutely continuous probability measure $\nu$ with respect to $\mu_\sigma$. Then provided the inequality makes sense, the following upper bound on the Rubinstein distance holds:
  \begin{equation}
    \label{eq:14}
    \T_{\rho }(\mu_\sigma, \nu) \leq \int _{\Gamma _\Lambda } \int_\Lambda \left \vert \int_0^{+\infty} \nabla _x P_t L (\omega ) \d t \right \vert \d \alpha (\omega , x) .
  \end{equation}
\end{proposition}
\begin{proof}
 The proof follows the approach emphasized by Houdr\'e and Privault in \cite{MR1962538} to derive covariance identities and then concentration inequalities.
  Letting $F \in \rho -\Lip_1$, we have by reversibility and using Fubini's Theorem:
  \begin{eqnarray*}
    \int _{\Gamma _\Lambda } F \d (\mu _{\sigma } -\nu ) &= & \int _{\Gamma _\Lambda } \left(\int _{\Gamma _\Lambda } F \d \mu _{\sigma } -F \right) L \d \mu _{\sigma } \\
    & = &\int _{\Gamma _\Lambda } \left( \int_0^{+\infty} \frac{\d}{\d t} P_t F \d t \right) L \d \mu _{\sigma } \\
    & = &- \int _{\Gamma _\Lambda } \int_0^{+\infty} P_t \L F \, L \d t \d \mu _{\sigma } \\
    & = &- \int _{\Gamma _\Lambda } \int_0^{+\infty} \delta \nabla F \, P_t L \d t \d \mu _{\sigma } \\
    & = &- \int _{\Gamma _\Lambda } \int _\Lambda \nabla _x F \int_0^{+\infty} \nabla _x P_t L \d t \d \alpha (\cdot , x) .
  \end{eqnarray*}
  Using then the Rademacher property (\ref{eq:rademacher}), the result holds by taking the supremum over all functions $F \in \rho -\Lip _1 $.
\end{proof}
\paragraph{}Note that the upper bound in the inequality (\ref{eq:14}) is interesting in its
own right, but seems to be somewhat difficult to compute in full generality. Hence we turn in the sequel to more concrete situations, i.e., when the gradient of interest is the discrete gradient $\GD$ or the differential one $\GC$ and is associated to the convenient distance $\rho _i$, $i \in  \{ 0,1,2\}$, in the sense of the Rademacher property (\ref{eq:rademacher}).
\subsection{A qualitative upper bound on $\T_{\rho _1}$}
\label{sec:rho_1-rubinst-dist}
Once the abstract estimate (\ref{eq:14}) has been obtained, one notices that it might be simplified whenever a commutation relation between gradient and semi-group holds. To the knowledge of the authors, such a property is only verified in the case of the discrete gradient, so that we focus in this part on the couple $(\GD , \rho _1)$. Here is one of the two main results of the paper.
\begin{theorem}
  \label{T:E1}
  Let $L$ be the density of an absolutely continuous probability
  measure $\nu$ with respect to $\mu_\sigma$, and assume that $L\in
  \Dom \GD$ and $\GD L \in L^1(\mu _\sigma \otimes \sigma ).$ Then we
  get the following estimate:
  \begin{align}\label{eq:14b}
    \T_{\rho _1}(\mu_\sigma, \nu) & \le \esp{\mu _{\sigma }}{
      \int_{\Lambda} \vert \GD_x L \vert \d \sigma(x)} .
  \end{align}
  The same inequality also holds under the distance $\rho _0$.
\end{theorem}
\begin{proof}
Since the case of a general intensity measure $\sigma \in \M (\Lambda )$ might be established by a simple limiting procedure (use the $\sigma$-compactness of the metric space $\Lambda$ and the lower semi-continuity of the application $\sigma \mapsto \T_{\rho _1}(\mu_\sigma, \nu)$), let us assume that $\sigma$ is finite, so that the Rademacher property stated in Lemma~\ref{lem:lip_discrete} is satisfied by the couple $(\GD , \rho _1)$. Hence Proposition~\ref{T:E0} above entails the inequality
\begin{eqnarray*}
\T_{\rho _1}(\mu_\sigma, \nu) & \leq & \esp{\mu_\sigma}{\int_{\Lambda} \left \vert \int_0^{+\infty} \GD_{x} \PD_{t} L \d t \right \vert \d \sigma (x)}.
\end{eqnarray*}
Using now the commutation relation (\ref{eqn:commutation}), we have:
\begin{eqnarray}
\label{eq:first_T1}
\T_{\rho _1}(\mu_\sigma, \nu) & \le & \esp{\mu_\sigma}{\int_{\Lambda} \left \vert \int_0^{+\infty}
        e^{-t}\PD_{t}\GD_{x}L \d t \right \vert \d \sigma (x)} \\
    \nonumber &\le & \esp{\mu _{\sigma }}{\int_{\Lambda} \int_0^{+\infty}
      e^{-t}\PD_{t}|\GD_{x}L| \d t \d \sigma(x)} \\
    \nonumber & = & \esp{\mu _{\sigma }}{\int_{\Lambda} \int_0^{+\infty}
      e^{-t}|\GD_{x}L| \d t \d \sigma(x)} \\
    \nonumber &= & \esp{\mu _{\sigma }}{ \int_{\Lambda} | \GD_x L | \d\sigma(x)},
  \end{eqnarray}
  where we have used Jensen's inequality and the invariance property
  of the Poisson measure $\mu _{\sigma }$ with respect to the
  semi-group $\PD_t$. The desired inequality (\ref{eq:14b}) is thus established.
\paragraph{}Finally, the case of the trivial distance $\rho _0$ is similar since the couple $(\GD , \rho _0)$ also satisfies the Rademacher property, cf. Lemma~\ref{lem:lip_discrete}. The proof is achieved in full generality.
\end{proof}
\paragraph{}Actually, the well-known relationship between semi-group and generator
states that for any $G\in L^2 (\mu _{\sigma } )$,
$$
\int_{0}^{+\infty}e^{-t}\PD_{t} G \d t=(\id + \LD)^{-1}G .
$$
Applying then such an identity in the inequality (\ref{eq:first_T1}) above gives the following bound:
\begin{equation}
  \label{eq:14t}
  \T_{\rho _1}(\mu _{\sigma }, \nu)\le \esp{\mu _{\sigma }}{\int_\Lambda |(\id
    +\LD)^{-1}\GD_x L|\d\sigma(x)} .
\end{equation}
It seems theoretically slightly better than the upper bound of
Theorem~\ref{T:E1} but often yields to intractable computations,
except when the chaos representation of $L$ is given, as noticed in
Section~\ref{sec:dist-betw-poiss} below. Note that the very analog of
\eqref{eq:14t} on Wiener space was proved by a different though
related way in Theorem~3.2 of \cite{MR2036490}. 
\paragraph{}
Let us provide another method leading to Theorem~\ref{T:E1} which is based on the so-called Clark formula. Instead of considering configurations in $\Gamma _\Lambda$, the idea is to use multivariate Poisson processes, i.e., point processes on $[0,1]$ with marks in the $\sigma$-compact metric space $\Lambda$. Borrowing an idea of \cite{Wu:2000lr}, we first explain how to embed a Poisson process into a  multivariate Poisson process.
\paragraph{}Let $\hmu$ be the Poisson measure of intensity $\lambda \otimes
\sigma$ on the new configuration space $\Gamma _{\hX}$, where the enlarged state space is $\hX = [0,1] \times \Lambda$, and $\lambda$ denotes the Lebesgue measure on $[0,1]$. Any generic element $\homega\in \Gamma _{\hX}$ has the form $\homega = \sum
_{(t,x) \in \homega} \ee_{t,x}$. The canonical filtration is defined for any $t\in [0,1]$ as
\begin{equation*}
  \F{t} = \sigma \left \{ \homega([0, s]\times B), \, \, 0\le s \le t, \, \,  B\in \B(\Lambda)\right \}.
\end{equation*}
Let us recall the Clark formula, cf. for instance \cite{MR960543} or
Lemma~1.3 in \cite{Wu:2000lr}, which states that every functional $G
:\Gamma _{\hX} \to \R$ belonging to $\Dom \GD$ might be written as
\begin{equation}
  \label{eq:clark}
  G = \esp{\hmu}{G} + \int_0^1 \int_\Lambda \esp{\hmu}{\GD _{t,x} G \, | \, \F{t^-}} \d ( \homega-\lambda \otimes \sigma )(t,x) ,
\end{equation}
where $\GD _{t,x}$ denotes the discrete gradient on the enlarged configuration space $\Gamma _{\hX}$.
\paragraph{}For an element $\homega\in \Gamma _{\hX}$, we define by $\pi \homega$
its projection on $\Gamma_\Lambda$, i.e.,
\begin{equation*}
  \pi \homega (B) = \homega ([0,1] \times B), \quad B \in \B (\Lambda ),
\end{equation*}
and given $F : \Gamma _\Lambda \to \R$, we define the functional $\hF$
as
\begin{align*}
  \hF : \Gamma_{\hX} & \longrightarrow \R \\
  \homega & \longmapsto F(\pi \homega) .
\end{align*}
In particular, we have clearly $\GD _{t,x} \hF (\homega) = \GD _x F (\pi \homega)$ for any $(t,x) \in
\hX$. Moreover, we have $\mathbb{E} _{\hmu } [\hF ] = \esp{\mu _{\sigma }}{F}$ since the image measure of $\hmu$ by $\pi$ is $\mu _{\sigma }$.
\paragraph{}The total variation distance on $\Gamma _{\hX}$ is defined as
$$
\hd2 (\homega ,\heta )= \sum _{(t,x)\in \hX} \vert \homega (\{ t,x \})-\heta (\{ t,x \} )\vert .
$$
The key point is the following lemma.
\begin{lemma}
  \label{lem:prolongement}
  For any $F \in \rho_1 -\Lip_{1}$, the functional $\hF$ belongs to
  $\hd2 -\Lip _1$.
\end{lemma}
\begin{proof}
  Given $F \in \rho_1 -\Lip_{1}$, we have for any $\homega , \heta \in
  \Gamma _{\hX}$:
  \begin{eqnarray*}
    \label{eq:3}
    \vert \hF(\homega)-\hF(\heta) \vert & = & \vert F(\pi\homega)-F(\pi\heta) \vert \\
    & \leq & \rho _1 (\pi \homega, \pi \heta) \\
    & = & \sum _{x\in \Lambda } \left \vert \pi \homega (\{ x \})-\pi \heta (\{ x \} ) \right \vert  \\
    & = & \sum _{x\in \Lambda } \left \vert \sum _{t\in [0,1]} \homega (\{ t,x \})- \heta (\{ t,x \}) \right \vert  \\
    & \leq & \sum _{(t,x)\in \hX} \vert \homega (\{ t,x \})-\heta (\{ t,x \} )\vert  \\
    & = & \hd2 (\homega ,\heta ) .
  \end{eqnarray*}
  The proof is complete.
\end{proof}
\paragraph{}Now we are able to give a second proof of Theorem~\ref{T:E1} by means
of the Clark formula (\ref{eq:clark}) and Lemma~\ref{lem:prolongement}.
\begin{proof}
  Letting $\hnu$ be the measure with density $\hL$ with respect to
  $\hmu$, we obtain:
  \begin{align*}
    \T_{\rho _1}(\mu_\sigma,\nu)&=\sup_{F\in \rho _1 -\Lip_{1}} \esp{\mu_\sigma}{F(L-1)} \\
    &=\sup_{F\in \rho _1 -\Lip_{1}} \mathbb{E} _{\hmu} [\hF (\hL-1)] \\
    &=\sup_{F \in \rho _1 -\Lip_{1}} \mathbb{E} _{\hnu} [\hF ] - \mathbb{E} _{\hmu} [\hF ] .
  \end{align*}
  Now using the Clark formula~(\ref{eq:clark}) and taking expectation
  with respect to $\hnu$,
  \begin{align*}
    \mathbb{E} _{\hnu}[\hF] &= \mathbb{E} _{\hmu}[\hF] +
    \esp{\hnu}{\int_0^1\int_\Lambda \esp{\hmu}{\GD _{t,x}\hF\, |\,
        \F{t^-}}\d ( \homega-\lambda \otimes \sigma )(t,x)}\\
    &= \mathbb{E} _{\hmu}[\hF] + \esp{\hmu}{\hL \int_0^1\int_\Lambda
      \esp{\hmu}{\GD _{t,x}\hF\, |\,
        \F{t^-}}\d ( \homega-\lambda \otimes \sigma )(t,x)}\\
    &= \mathbb{E} _{\hmu}[\hF] + \esp{\hmu}{\int_0^1\int_\Lambda \esp{\hmu}{\GD
        _{t,x}\hF\, |\, \F{t^-}}\GD _{t,x}\hL \d t \d \sigma(x)},
  \end{align*}
  where in the second line we also used the Clark formula~(\ref{eq:clark}) applied to the functional $\hL$. By Lemma~\ref{lem:lip_discrete}, the couple $(\GD , \hd2)$ satisfies the Rademacher property (\ref{eq:rademacher}) on $\Gamma _{\hX}$. Hence Lemma~\ref{lem:prolongement} implies that for $F\in \rho _1 -\Lip_1$, the quantity $ \left \vert \esp{\hmu}{\GD_{t,x}\hF\, | \, \F{t^-}} \right \vert$ is bounded by $1$, $\hmu \otimes \lambda \otimes \sigma$-a.e., so that we obtain finally
  \begin{eqnarray*}
    \T_{\rho _1}(\mu _{\sigma } ,\nu) & \le &\esp{\hmu}{\int_0^1 \int_\Lambda |\GD_{t,x}\hL| \d t \d \sigma(x)} \\
    & = & \esp{\mu_\sigma}{\int_\Lambda |\GD_{x}L| \d \sigma(x)}.
  \end{eqnarray*}
  The second proof of Theorem~\ref{T:E1} is thus complete.
\end{proof}
\subsection{A qualitative upper bound on $\T_{\rho _2}$ by time-change}
\label{sec:time_change}
Recall that by Lemma~\ref{lem:lip_wasserstein}, the couple $(\GC , \rho _2)$ satisfies the Rademacher property (\ref{eq:rademacher}). Hence Proposition~\ref{T:E0} entails an upper bound on the $\T_{\rho _2}$ Rubinstein distance as follows: if $L$ denotes the density of an absolutely continuous probability measure $\nu$ with respect to $\mu_\sigma$, then we have
\begin{eqnarray*}
\T_{\rho _2}(\mu_\sigma, \nu) & \leq & \int _{\Gamma _\Lambda } \int_\Lambda \left \vert \int_0^{+\infty} \GC _x \PC _t L (\omega ) \d t \right \vert \d \omega (x) \, \d \mu _{\sigma } (\omega ) ,
\end{eqnarray*}
provided the inequality makes sense. However, despite its theoretical interest, such an inequality is not really tractable in practise, since no commutation relation has been established yet between the differential gradient $\GC$ and the semi-group $\PC _t$. Hence the purpose of this section is to provide another estimate on $\T_{\rho _2}$ through a different approach relying on a time-change argument together with the Girsanov Theorem.
\paragraph{}We consider the notation of Section~\ref{sec:rho_1-rubinst-dist} above, with the difference that the state space is now $\hX := [0, \infty) \times \Lambda $, where $\Lambda$ is the space $\R ^k$ equipped with the Euclidean distance $\kappa$. In this part, the distance of interest on the enlarged configuration space $\Gamma _{\hX}$ is the Wasserstein distance:
$$
\hrho (\homega , \heta )^2= \inf _{\beta\in \Sigma({\homega , \heta
  })} \, \int _{\hX} \int _{\hX} (\kappa (x,y) ^2 + \vert t-s\vert ^2
) \d \beta ((s,x),(t,y)).
$$
The following theorem is our second main result.
\begin{theorem}
  \label{thm:time_change}
  Let $L$ be the (positive) density of an absolutely continuous probability measure $\hnu$ with
  respect to $\hmu $. Then provided the inequality makes sense, we get the following upper bound on the Rubinstein distance $\T_{\hrho}(\hmu , \hnu)$:
  \begin{equation}\label{eq:6}
    \begin{split}
    \T_{\hrho}(\hmu , \hnu)^2 & \le \esp{\hmu }{L \int _\Lambda \int
      _0 ^{+\infty} \left \vert \int _0 ^t u(s,z) \d  s \right \vert
      ^2 (1+u(t,z) ) \d t \d \sigma (z)} \\
&= \esp{\hmu }{L\int _\Lambda \int _0 ^{+\infty} \left \vert r-v^{-1}(r,z) \right \vert ^2 \d r \d \sigma (z)}    ,
    \end{split}
    \end{equation}
where $u(t,z) >-1$ is the following predictable process:
$$
u(t,z) = \frac{\esp{}{\GD _{t,z} L | \F{t^-}}}{\esp{}{L | \F{t^-}}},\quad v (t,z) := t+\int _0 ^t u(s,z) \d s , \quad z\in \Lambda ,
$$
and $v^{-1}(\cdot , z)$ is the inverse of the increasing mapping $t\mapsto v(t,z)$.
\end{theorem}
\paragraph{}
Note that for $z\in \Lambda $ fixed, the term $ \int _0 ^{+\infty} \left \vert r-v^{-1}(r,\,z) \right \vert ^2 \d r$ can be interpreted as a generalized Wassertein distance between the infinite measures $\d r$ and $(1+u(r,z))\d r$, see \cite{Villani:2007fk}. Then, the $\T _{\hrho}$ distance is bounded from above by the expectation under $\hnu$ of this generalized distance integrated over $\Lambda$ according to the marks distribution.
\begin{proof}
By the Girsanov Theorem, there exists a predictable process $u$ such that for any compact set $K\in \B(\Lambda)$, the process
\begin{equation*}
t\mapsto \homega ([0, t ]\times K)-\int_0^t \int_K (1+u(s, z)) \d s \d \sigma (z) ,
\end{equation*}
is a $\hnu$-martingale. Moreover, the conditional expectation $L_t:= \esp{}{L | \F{t}}$ might be identified as follows:
\begin{eqnarray*}
L_t & = & \exp \left \{\int _0 ^t \int_\Lambda \ln (1+u(s,z)) \d \homega (s,z) -\int _0 ^t \int _\Lambda u(s,z) \d s \d \sigma (z) \right \} \\
& = & \mathcal{E}\left(\int _0 ^t \int_\Lambda  u(s, z) \d (\homega - \lambda \otimes \sigma )(s,z) \right) \\
& = & 1+\int_0^t \int_\Lambda L_{s^-} u(s,z) \d (\homega - \lambda \otimes \sigma )(s,z),
\end{eqnarray*}
where $\mathcal{E}$ denotes the classical Dol\'eans-Dade exponential. On the other hand, the Clark formula~(\ref{eq:clark}) extended to the set $(0,+\infty)$ induces that
\begin{equation*}
L_t = 1+\int_0^t \int_\Lambda \mathbb{E} \left[ \GD_{s,z} L_t| \F{s^-} \right] 
\d (\homega - \lambda \otimes \sigma )(s,z).
\end{equation*}
By identification, we obtain:
\begin{equation*} 
u(s,z) = \frac{\mathbb{E} \left[ \GD_{s,z} L_t| \F{s^-} \right]}{L_{s^-}} 
= \frac{\mathbb{E} \left[ \GD_{s,z} L| \F{s^-} \right]}{L_{s^-}}, 
\end{equation*}
since for any $s\in (0,t)$ a commutation relation holds between the discrete grad\-ient $\GD _{s,z}$ and the conditional expectation knowing $\F{t}$, cf.~for instance Lemma~3.2 in \cite{nualart88_1}. Define on $\Gamma_{\hX}$ the time-change configuration $\tau \homega$ by
$$
\tau \homega = \sum _{(t_i,z_i) \in \homega } \ee _{v(t_i, z_i), z_i},
$$
where $v(t,z)$ is given above. By Theorem~3 in \cite{these_decreusefond}, the distribution of $\tau \homega$ under $\hnu$ is nothing but the law of the configuration $\homega$ under $\hmu$.~Hence using Cauchy-Schwarz' inequality in the second line below, we obtain:
\begin{eqnarray*}
\T_{\hrho }(\hmu , \hnu) & \leq & \esp{\hnu }{\hrho (\homega , \tau \homega )} \\
& \leq & \esp{\hnu }{\int _\Lambda \int _0 ^{+\infty} \vert t - v(t,z) \vert ^2 \d \homega (t,z)}^{1/2} \\
& = & \esp{\hnu }{ \int _\Lambda \int _0 ^{+\infty} \left \vert t-v(t,z) \right \vert ^2 \frac{\d v}{\d t}(t,z) \d t \d \sigma (z)}^{1/2} ,
\end{eqnarray*}
where we used the classical compensation formula for stochastic integrals with respect to Poisson random measures. Finally, the change of variable $r=v(t,z)$ for $z\in \Lambda $ being fixed allows us to obtain the desired inequality (\ref{eq:6}).
\end{proof}
\section{Applications}
\label{sec:applications}
\subsection{Distance estimates between processes}
\label{sec:dist-betw-poiss}
The purpose of the present part is to apply our main results Theorems~\ref{T:E1} and \ref{thm:time_change} to provide distance estimates between a Poisson process and several other more sophisticated processes, such as Cox or Gibbs processes. See for instance the pioneer monograph \cite{MR93g:60043} or also \cite{MR1708412,Schuhmacher:2005lx} for similar results with respect to another (bounded) distances on the configuration space $\Gamma _\Lambda$. The three first examples below rely on the total variation distance $\rho _1$, whereas in the last one the Wasserstein distance $\rho _2$ is considered. 
\paragraph{} {\bf \textit{Poisson processes.}} Here the probability measure $\nu$ is supposed to be another Poisson measure on $\Gamma_\Lambda$, where $\Lambda$ is a $\sigma$-compact metric space.
\begin{proposition}
  \label{thm:dist-betw-poiss}
  Let $\mu_\tau$ be a Poisson measure on $\Gamma_\Lambda$ of intensity
  $\tau$. We assume that $\tau$ admits a density $p$ with
  respect to $\sigma$ such that $p-1 \in L^1 (\sigma)$. Then we have
  \begin{equation}
  \label{eq:poisson_process}
    \T_{\rho_1}(\mu_\sigma, \mu_\tau)\le \int _\Lambda |p(x)-1| \d  \sigma(x).
  \end{equation}
\end{proposition}
\begin{proof}
  Since $\mu_\tau$ is a Poisson measure on $\Gamma_\Lambda$ of
  intensity $\tau$, it is well known that it is absolutely continuous
  with respect to $\mu_\sigma$ and the density $L$ is given by
  $$
    L(\omega )= \exp \left \{ \int _\Lambda \log p(x) \d\omega (x) +\int _\Lambda (1-p(x))\d\sigma (x) \right \}.
  $$
  It is then straightforward that $\GD_x L = L ( p(x)-1)$, hence by
  Theorem~\ref{T:E1},
$$
\T_{\rho_1}(\mu_\sigma, \mu_\tau)\le \esp{\mu _\sigma }{L\int _\Lambda
  |p(x)-1|\d\sigma(x)}=\int _\Lambda |p(x)-1|\d\sigma(x).
$$
The proof is achieved.
\end{proof}
\paragraph{}Note that in this very simple situation, the inequality (\ref{eq:14t})
yields the same bound. Indeed, since $p$ is deterministic, the
density $L$ has the following chaos representation
\begin{equation*}
  L=1+\sum_{n=1}^{\infty} \frac{1}{n!}J_n \left( (p-1)^{\otimes n}\right) ,
\end{equation*}
cf. identity (7) in \cite{ruiz85}, so that we have
\begin{equation*}
  ((\id +\LD)^{-1}\GD_x L=(p(x)-1) \sum_{n=1}^{\infty} \frac{1}{(n-1)!} J_{n-1}  \left( (p-1)^{\otimes n-1}\right) = (p(x)-1)L.
\end{equation*}
Actually, one might obtain the inequality (\ref{eq:poisson_process}) by using another very intuitive approach. Indeed, let $\omega _0$, $\omega _1$ and $\omega _2$ be three independent configurations in $\Gamma _\Lambda$ with respective intensities
$$
\d \sigma _0  := (p \wedge 1 ) \d \sigma , \quad \sigma _1 := \sigma - \sigma _0 , \quad \sigma _2 := \tau - \sigma _0 .
$$
Then $\omega _0 + \omega _1$ and $\omega _0 + \omega _2$ have respective distribution $\mu _\sigma$ and $\mu _\tau$. Hence we have
\begin{eqnarray*}
\T_{\rho_1}(\mu_\sigma, \mu_\tau) & = & \inf \left \{ \esp{}{\rho _1 (\omega , \bar{\omega})} : \omega \sim \mu_\sigma , \, \bar{\omega}\sim \mu_\tau \right \} \\
& \leq & \esp{}{\rho _1 (\omega _0 + \omega _1, \omega _0 + \omega _2 )} \\
& = & \esp{}{(\omega _1 + \omega _2 )(\Lambda )} \\
& = & \int _\Lambda |p(x)-1|\d\sigma(x).
\end{eqnarray*}
\paragraph{}{\bf \textit{Cox processes.}}
A Cox process is a Poisson process with a random intensity. To
construct a Cox process, we need to enlarge our probability
space. Recall that $\M(\Lambda)$ is the space of positive and diffuse Radon measures on
$\Lambda$ endowed with the vague topology and the corresponding Borel
$\sigma$-field. Given an arbitrary probability measure $\P _M$ on $\M(\Lambda)$, we denote by $M$ the canonical random variable on
$(\M(\Lambda), \P _M)$, i.e. $M$ given by $M(m)=m$ has distribution $\P _M$. On the space
$\Gamma_\Lambda\times \M(\Lambda)$, we consider the probability measures
$$
\d \mu_M^\prime(\omega, m) := \d\mu_m(\omega)\d\P _M(m) \quad \mbox{and} \quad \d\mu_\sigma^\prime(\omega, m) := \d\mu_\sigma(\omega)\d\P_M(m) .
$$
Note that the second one is the distribution of the independent couple $(N,M)$, where $N$ is the canonical random variable on $\Gamma _\Lambda$ with distribution $\mu _\sigma$.
\paragraph{}As noticed in Section~\ref{sec:distances}, the application $m\mapsto \T_{\rho_1}(\mu_m , \mu_\sigma)$ is lower semi-con\-ti\-nuous, hence
measurable.  The distribution $\mu_M^\prime$ on $\Gamma_\Lambda$ is
said to be Cox whenever for any function $f\in {\mathcal C}_0 (\Lambda )$,
\begin{equation*}
  \esp{\mu_M^\prime}{\exp\left( \int_\Lambda f\d\omega \right) \, \biggl|\,
    M}=\exp\left\{\int_\Lambda (e^f-1)\d M\right\}.
\end{equation*}
In the definition of the distance between $\mu_M^\prime$ and $\mu^\prime
_\sigma $, we do not include any information on $M$, so that the distance $\rho _1$ remains the same and we have:
\begin{eqnarray*}
  \T_{\rho_1}(\mu^\prime_\sigma, \mu_M^\prime)& =&\sup_{F\in
    \rho_1-\Lip _1} \int _{\Gamma_\Lambda\times \M(\Lambda) } F(\omega) \d \mu_\sigma ^\prime (\omega,
  m)-\int _{\Gamma_\Lambda\times \M(\Lambda) } F(\omega) \d \mu_M ^\prime(\omega, m) \\
  & = & \sup_{F\in \rho_1-\Lip _1} \int_{\M(\Lambda)} \left(\int _{\Gamma _\Lambda } F (\omega)
    \d (\mu_\sigma - \mu_m )(\omega)   \right)\d\P_M(m).
\end{eqnarray*}
\begin{proposition}
  Assume that $\mu_\sigma^\prime$-a.s., the measure $M$ is absolutely
  continuous with respect to~$\sigma$ and that there exists a
  measurable version of $\d   M / \d  \sigma$ and such that $\d   M / \d  \sigma - 1 \in L^1 (\mu _\sigma^\prime  \otimes \sigma )$. Then we have
  \begin{equation*}
    \T_{\rho_1}(\mu_\sigma^\prime, \mu_M^\prime)\le \esp{\mu _\sigma^\prime }{\int_\Lambda
      \left|\frac{\d   M}{\d   \sigma}(x)-1\right| \d  \sigma(x)}.
  \end{equation*}
\end{proposition}
\begin{proof}
We have:
  \begin{align*}
    \T_{\rho_1}(\mu^\prime_\sigma, \mu_M^\prime) &\le
    \int_{\M(\Lambda)} \sup_{F\in \rho_1 - \Lip _1} \left(\int
      _{\Gamma _\Lambda }F(\omega) \d (\mu_\sigma - \mu_m )(\omega)
    \right) \d\P_M(m)\\
    &= \int_{\M(\Lambda)} \T_{\rho_1}(\mu_\sigma, \mu_m) \d\P_M(m)\\
    &\le \int_{\M(\Lambda)} \int_\Lambda \left| \frac{\d m}{\d
        \sigma}(x)-1\right|\d\sigma(x) \d\P_M(m),
  \end{align*}
  where the last inequality follows from Proposition~\ref{thm:dist-betw-poiss}.
\end{proof}
\paragraph{} {\bf \textit{Gibbs processes.}}
Let $\Lambda = \R ^k$ and assume that the measure $\nu$ is a Gibbs
measure on $\Gamma_\Lambda$ with respect to the reference measure $\mu
_\sigma$, i.e. the density of $\nu$ with respect to $\mu_\sigma$ is of the form $L=e^{-V}$, where
$$
V(\omega) := \int _\Lambda \int _\Lambda \phi(x-y)\d\omega(x)\d\omega(y) < +\infty , \quad \mu_\sigma -a.s.,
$$
and where the potential $\phi : \Lambda \to (0,+\infty )$ is such that $\phi(x)=\phi(-x)$ and
$$
\int _\Lambda \int _\Lambda \phi(x-y)\d  \sigma(x)\d  \sigma(y) <+\infty .
$$
We have the following result.
\begin{proposition}
The Rubinstein distance $\T_{\rho_1}$ between the Poisson measure $\mu_\sigma$ and the Gibbs measure $\nu$ is bounded as follows:
  \begin{equation*}
    \T_{\rho _1}(\mu_\sigma, \nu)\le 2\, \int _\Lambda \int _\Lambda \phi(x-y)\d  \sigma(x)\d  \sigma(y) .
  \end{equation*}
\end{proposition}
\begin{proof}
  Since $V$ is $\mu_\sigma$-a.s. finite, so does $\int _\Lambda \phi(x-y)\d\omega(y)$ for any $x$. We have:
  \begin{equation*}
    \GD_x L (\omega ) = - L (\omega) \left( 1-\exp \left \{ - 2\int _\Lambda \phi(x-y)\d\omega(y) \right \} \right), \quad x\in \Lambda.
  \end{equation*}
  Since $0\le L\le 1$, Theorem \ref{T:E1} together with the inequality $1-e^{-u}\le u$ imply:
  \begin{eqnarray*}
    \T_{\rho _1}(\mu_\sigma, \nu) &\leq & \esp{\mu_\sigma}{L \int_\Lambda \left( 1-\exp \left \{ - 2\int _\Lambda \phi(x-y)\d\omega(y)\right\}\right) \d\sigma(x) } \\
  &\leq & \esp{\mu_\sigma}{L \int_\Lambda 2\int _\Lambda \phi(x-y) \d\omega(y) \d\sigma(x)}\\
  &\leq &  2 \,\esp{\mu_\sigma}{\int_\Lambda \int_\Lambda \phi(x-y)\d\omega(y)\d\sigma(x)}\\
  & = & 2 \, \int_\Lambda \int_\Lambda \phi(x-y) \d\sigma(x)\d\sigma(y) .
  \end{eqnarray*}
  The proof is complete.
\end{proof}
\paragraph{} {\bf \textit{Poisson processes on the half-line.}} In this example, we give a bound on the Rubinstein distance between Poisson processes, with respect to the Wasserstein distance $\rho _2$. Consider to simplify Poisson processes on $\R _+$ (the generalization to multivariate Poisson processes is straightforward). Letting $U : \R _+ \to \R $ be a continuously differentiable function vanishing at infinity and with $U(0)=0$, we also assume that $U\in L^2 (\lambda ),$ where $\lambda$ is the Lebesgue measure, and that its derivative $U'$ is valued in $(-1,+\infty)$. A typical example of such a function is $U(t) = t/(1+t^3)$, $t\geq 0$. Then we obtain by Theorem \ref{thm:time_change} the following result.
\begin{proposition} Let $\mu _\lambda $ be the Poisson measure of Lebesgue intensity $\lambda$ on the configuration space $\Gamma _{\R _+}$, and consider the Poisson measure $\nu $ of intensity $(1+U') \d \lambda$. Then we have the upper bound on $\T_{\rho _2}(\mu_\lambda , \nu )$:
\begin{eqnarray*}
\T_{\rho _2}(\mu_\lambda , \nu ) & \leq & \left \Vert U \right \Vert _{L^2 (\lambda )} .
\end{eqnarray*}
\end{proposition}
\subsection{Tail and isoperimetric estimates}
The aim of this final part is to derive several consequences of Theorem~\ref{T:E1} above in terms of tail estimates and isoperimetric inequalities. \paragraph{} {\bf \textit{Tail estimates.}} Our main result Theorem \ref{T:E1} allows us to obtain a first tail estimate as follows. Let $F \in \rho _1 -\Lip _1$ be centered and let $\lambda >0$. Denote $Z_\lambda = \esp{{\mu_\sigma} }{e^{\lambda F
  }}$ and consider $\nu ^\lambda$ the absolutely continuous probability measure with density $e^{\lambda F } /Z_\lambda $ with
respect to ${\mu_\sigma}$. Using a somewhat similar argument as in \cite{MR1962538}, we have:
\begin{eqnarray*}
  \frac{\d }{\d \lambda } \log Z_ \lambda & = & \int _{\Gamma _\Lambda } F \d \nu ^{\lambda } \\
  & \leq & \T _{\rho _1} ({\mu_\sigma} , \nu ^\lambda ) \\
  & \leq & \esp{{\mu_\sigma}}{ \int _\Lambda \vert \GD _x e^{\lambda F} \vert \d \sigma (x)} \\
  & \leq & (e^\lambda -1) \, \Vert \GD F \Vert _{1,\infty} ,
\end{eqnarray*}
where in the last inequality we used the fact that the function
$x\mapsto (e^x -1)/x$ is non-decreasing on $(0,+\infty )$. Here the notation $\Vert \GD F \Vert _{1,\infty}$ stands for
$$
\Vert \GD F \Vert _{1,\infty} := \mu _\sigma - \esssup \, \int _\Lambda \vert \GD _x F\vert \d \sigma (x).
$$
Hence we obtain the following bound on the Laplace transform:
$$
\esp{\mu _\sigma }{e^{\lambda F }} = Z_\lambda \leq \exp \left \{
  \Vert \GD F \Vert _{1,\infty} \, (e^\lambda - \lambda
  -1 )\right \} , \quad \lambda >0.
$$
Finally using Chebychev's inequality, we get the deviation inequality available for any $r\geq 0$:
\begin{equation}
\label{eq:tail}
  {\mu_\sigma} \left( F\geq r \right) \leq \exp \left \{ r - (r+ \Vert \GD F \Vert _{1,\infty}) \, \log \left( 1+ \frac{r}{\Vert \GD F\Vert _{1,\infty} }\right) \right \} .
\end{equation}
\paragraph{}Note that such a tail estimate is somewhat similar to that established for instance by Wu and Houdr\'e-Privault in \cite{Wu:2000lr, MR1962538}. However, in contrast to their results, we do not exhibit at the denominator the sharp variance term
$$
\Vert \GD F \Vert _{2,\infty} ^2 := \mu _\sigma - \esssup \, \int _\Lambda \vert \GD _x F\vert ^2 \d \sigma (x),
$$
since our method relies on the $L^1$-inequality (\ref{eq:14b}). In particular, if we apply (\ref{eq:tail}) for instance to the centered function $F\in \rho _1 -\Lip _1 $ given by $F(\omega ) = (\omega -\sigma)(K)$, where $K$ is some compact subset of $\Lambda$, we obtain the inequality
$$
{\mu_\sigma} \left( \omega (K) \geq \sigma (K)+r \right) \leq e^{r - (r+ \sigma (K)) \, \log \left( 1+ \frac{r}{\sigma (K)}\right)} .
$$
Unfortunately, neither (\ref{eq:tail}) nor the results emphasized in \cite{Wu:2000lr, MR1962538} are sharp in terms of the deviation level $r$ since the following asymptotic estimate holds, cf. for instance p.1225 of Houdr\'e \cite{houdre}:
\begin{eqnarray*}
{\mu_\sigma} \left( \omega (K) \geq \sigma (K)+r \right) & = & {\mu_\sigma} \left( \omega (K) \geq [\sigma (K)+r] \right)  \\
& \underset{r\to +\infty}{\sim} & \frac{e^{[\sigma (K) +r] - \sigma (K) - [\sigma (K)+r] \, \log \left( \frac{[\sigma (K) + r]}{\sigma (K)}\right)} }{\sqrt{2\pi [\sigma (K) +r]}},
\end{eqnarray*}
where $[R]:= \inf \{ N\in \N _*: N\geq R\}$ denotes the upper integer part of any positive real number $R$. Hence the purpose of this part is to recover this multiplicative polynomial factor by means of a simple use of Theorem~\ref{T:E1}. We proceed as follows. Let $\nu$ be the absolutely continuous probability measure with density with respect to $\mu _\sigma$:
$$ L := \frac{1}{{\mu_\sigma} \left( \omega (K) \geq [ \sigma(K) + r] \right) } \, \car _{ \{ \omega (K) \geq  [ \sigma(K) + r] \}} , \quad r>0 .$$ Using Theorem~\ref{T:E1}, we compute as follows:
\begin{eqnarray*}
\lefteqn{ \mu_\sigma \left( \omega (K) \geq \sigma (K)+r \right) } \\
& & = \mu_\sigma \left( \omega (K) \geq [\sigma (K)+r]  \right) \\
& & \leq \frac{1}{[\sigma (K)+r]} \, \left( \int _{\Gamma _\Lambda} \omega (K) \, L(\omega) \d \mu _\sigma (\omega) \right) {\mu_\sigma} \left( \omega (K) \geq [\sigma (K)+r]  \right) \\
& & \leq \frac{1}{[\sigma (K)+r]} \, \left( \vphantom{\biggl(} \T _{\rho _1} ({\mu_\sigma} ,\nu ) + \sigma (K)  \vphantom{\biggl)}\right) {\mu_\sigma} \left( \omega (K) \geq [\sigma (K) +r]  \right) \\
& & \leq \frac{1}{[\sigma (K)+r]} \, \left( \esp{{\mu_\sigma}}{ \int _\Lambda \vert \GD _x L \vert \, \d \sigma (x) } +\sigma (K) \right) {\mu_\sigma} \left( \omega (K)\geq [\sigma (K) + r]  \right) \\
& & = \frac{\sigma (K)}{[\sigma (K)+r]} \, \left( \vphantom{\biggl(} \mu_\sigma \left( \omega (K) = [\sigma (K) + r] -1\right)  + \mu_\sigma \left( \omega (K) \geq [\sigma (K) + r] \right) \vphantom{\biggl)}\right) ,
\end{eqnarray*}
so that we obtain
\begin{eqnarray*}
\mu_\sigma \left( \omega (K) \geq \sigma (K)+ r \right) & \leq & \frac{[\sigma (K)+r]}{[\sigma (K)+r] - \sigma (K)} \, e^{-\sigma (K)} \, \frac{\sigma (K) ^{[\sigma (K)+r]}}{[\sigma (K)+r]!} \\
& \leq & \frac{[\sigma (K)+r]}{r} \, e^{-\sigma (K)} \, \frac{\sigma (K) ^{[\sigma (K)+r]}}{[\sigma (K)+r]!} .
\end{eqnarray*}
Hence using the lower bound below on the factorial function of any positive integer $N$, cf. for instance \cite{feller}:
\begin{equation}
\label{eq:factoriel}
\sqrt{2\pi } \, N^{N+\frac{1}{2}} \, e^{-N} \, \leq \, N! \, \leq \,  \sqrt{2\pi } \, N^{N+\frac{1}{2}} \, e^{-N+ \frac{1}{12N}} ,
\end{equation}
we obtain the following result.
\begin{proposition}
Given any compact set $K\subset \Lambda$ and any $r>0$, we have the tail estimate:
\begin{eqnarray*}
\mu_\sigma \left( \omega (K) \geq \sigma (K)+r \right) & \leq & \frac{[\sigma (K)+r]}{r}\, \frac{e^{ [\sigma (K)+r] - \sigma (K) - [\sigma (K)+r] \, \log \left( \frac{[\sigma (K) +r]}{\sigma (K)}\right) }}{\sqrt{2\pi [\sigma (K) +r]}} .
\end{eqnarray*}
\end{proposition}
\paragraph{}To the knowledge of the authors, although the latter non-asymptotic tail estimate is straightforward to establish via Theorem~\ref{T:E1} as we have seen above, it seems to be new and recovers exactly the asymptotic regime emphasized above. Note that Paulauskas obtained a somewhat similar deviation inequality in Proposition~3 in \cite{paulauskas}, but with a constant which is however not sharp, in contrast to ours.
\paragraph{}Now we aim at extending this tail estimate to a more general context. Given a fixed configuration $\eta \in \Gamma _\Lambda$, we provide in the sequel a deviation inequality from its mean of the total variation distance $\rho _1$ between $\eta$ and random configurations. Assume that $\sigma$ is a finite measure. Denoting the function $\rho _\eta := \rho _1 (\cdot, \eta)$ which clearly belongs to the set $\rho _1 -\Lip _1 $ and using the same argument as above, we have
\begin{multline*}
 \mu_\sigma  \left(  \rho _\eta \geq  \esp{{\mu_\sigma}}{\rho _\eta } + r \right) \\
  \shoveleft{\quad \quad = \mu_\sigma  \left(  \rho _\eta \geq [ \esp{{\mu_\sigma}}{\rho _\eta } +r] \right) }\\
 \shoveleft{\quad \quad \leq  \frac{1}{[ \esp{{\mu_\sigma}}{\rho _\eta } +r] } \, \esp{\mu _\sigma }{\rho _\eta \, \car _{ \{\rho _\eta \geq [ \esp{{\mu_\sigma}}{\rho _\eta } +r] \} }} }\\
 \shoveleft{\quad \quad \leq \frac{1}{[ \esp{{\mu_\sigma}}{\rho _\eta } +r] } \, \left( \esp{{\mu_\sigma}}{ \int _\Lambda \vert \GD _x \car _{\{ \rho _\eta \geq [ \esp{{\mu_\sigma}}{\rho _\eta } +r]  \} } \vert \, \d \sigma (x) }\right.} \\
\shoveright{\left.+ \esp{\mu _\sigma }{\rho _\eta } \mu_\sigma  \left(  \rho _\eta \geq [ \esp{{\mu_\sigma}}{\rho _\eta } +r] \right)\vphantom{\biggl)}\right) }\\
 \shoveleft{\quad \quad \leq  \frac{[\sigma (\Lambda )+r]-r}{[ \esp{{\mu_\sigma}}{\rho _\eta } +r]} \, \left(\vphantom{\biggl(} \mu_\sigma  \left(  \rho _\eta  \geq [ \esp{{\mu_\sigma}}{\rho _\eta } +r-1] \right) - \mu_\sigma  \left(  \rho _\eta  \geq [ \esp{{\mu_\sigma}}{\rho _\eta } +r] \right) \vphantom{\biggl)}\right) }\\
+ \frac{1}{[ \esp{{\mu_\sigma}}{\rho _\eta } +r]} \, \esp{\mu _\sigma }{\rho _\eta } \, \mu_\sigma  \left(  \rho _\eta \geq [ \esp{{\mu_\sigma}}{\rho _\eta } +r] \right) ,
\end{multline*}
since the intensity measure $\sigma$ is diffuse. Hence we obtain for any $r>0$:
$$
\mu_\sigma  \left(  \rho _\eta \geq [ \esp{{\mu_\sigma}}{\rho _\eta } +r] \right) \leq \frac{[\sigma (\Lambda )+r]-r}{[\sigma (\Lambda ) + r]} \, \mu_\sigma  \left(  \rho _\eta \geq [ \esp{{\mu_\sigma}}{\rho _\eta } +r-1] \right) ,
$$
and iterating the procedure entails the inequality
$$
\mu_\sigma  \left(  \rho _\eta \geq [ \esp{{\mu_\sigma}}{\rho _\eta } +r]  \right) \leq \frac{\left([\sigma (\Lambda )+r]-r\right) ^{r} [\sigma (\Lambda )] ! }{[\sigma (\Lambda ) + r]!}.
$$
Finally using the estimates (\ref{eq:factoriel}) yield the following result.
\begin{proposition}
Given any fixed configuration $\eta \in \Gamma _\Lambda$ and provided the intensity measure $\sigma$ is finite, we have for any $r>0$:
\begin{multline*}
{\mu_\sigma \left( \rho _\eta \geq \esp{\mu _\sigma }{\rho _\eta } + r \right) } \\
\leq \frac{\sqrt{2\pi [\sigma (\Lambda)]} [\sigma (\Lambda)] ^{[\sigma (\Lambda)]} e^{\frac{1}{12 [\sigma (\Lambda)]}}} {\sigma (\Lambda)^{\sigma (\Lambda)}} \, \frac{e^{[\sigma (\Lambda)+r]- [\sigma (\Lambda)] - [\sigma (\Lambda)+r] \, \log \left( \frac{[\sigma (\Lambda)+r]}{[\sigma (\Lambda)+r]-r}\right)}} {\sqrt{2\pi [\sigma (\Lambda) +r]}},
\end{multline*}
where $\rho _\eta$ denotes the total variation distance $\rho _1 (\cdot , \eta )$.
\end{proposition}
\paragraph{} Hence one deduces that the tail behavior of the total variation distance is comparable to the previous ones, up to constant multiplicative factors depending on the total mass $\sigma (\Lambda)$.
\paragraph{}{\bf \textit{Isoperimetric inequality.}}
\label{sec:isop-ineq} Here the distance of interest is the trivial distance $\rho _0$. In the sequel, we assume that the intensity measure $\sigma$ is finite, so that the domain $\Dom \GD$ contains the indicator functions $\car _A$, $A \in {\mathcal
  B}(\Gamma_{\Lambda})$.
  \paragraph{}Given a Borel set $A \in {\mathcal B}(\Gamma_{\Lambda})$, we define its surface measure as
$$
{\mu_\sigma} (\partial A) := \esp{{\mu_\sigma}}{\int _\Lambda \vert
  \GD _x \car _A \vert \, \d \sigma (x)}.
$$
Denote $h_{\mu_\sigma}$ the classical isoperimetric constant that we aim at estimating:
$$
h_{\mu_\sigma} = 2 \, \inf _{0<{\mu_\sigma} (A)<1} \,
\frac{{\mu_\sigma} (\partial A)}{{\mu_\sigma} (A)(1-{\mu_\sigma} (A)
  )}.
$$
By the following co-area formula, available for any $F \in \Dom \GD$:
$$
\esp{{\mu_\sigma}}{ \int _\Lambda \vert \GD _x F \vert \, \d \sigma
  (x) } = \esp{{\mu_\sigma}}{ \int _\Lambda \int _{-\infty } ^{+\infty
  } \vert \GD _x \car _{\{ F >t \}} \vert \, \d t \, \d \sigma (x)},
$$
which might be deduced from the identity $\vert a-b\vert =
\int _{-\infty } ^{+\infty } \vert \car _{\{ a>t \}} - \car _{\{ b>t \}} \vert \, \d t$, the constant $h_{\mu_\sigma}$ is also the
best constant $h$ in the $L^1$-type functional inequality
\begin{equation}
  \label{eq:poincare_L1}
  h \, \esp{{\mu_\sigma} }{\left \vert F - \esp{{\mu_\sigma}}{F} \right \vert } \leq 2 \, \esp{{\mu_\sigma}}{ \int _\Lambda \vert \GD _x F \vert \, \d \sigma (x)} , \quad F \in \Dom \GD .
\end{equation}
We have the following result, which is convenient for small total mass $\sigma (\Lambda )$.
\begin{proposition}
  \label{theo:isoperimetrie}
  Assume that the measure $\sigma$ is finite. Then we have
  \begin{equation}
    \label{eq:isop}
    1 \leq h_{\mu_\sigma} \leq \frac{\sigma (\Lambda )}{1-e^{-\sigma (\Lambda )}}.
  \end{equation}
  In particular, we have the asymptotic for small total mass:
  $$
  \lim _{\sigma (\Lambda ) \to 0} h_{\mu_\sigma} =1.
  $$
\end{proposition}
\paragraph{}Note that Houdr\'e and Privault established first the inequality
$h_{\mu_\sigma} \geq 1$ by using Poincar\'e inequality,
cf. Proposition~6.4 in \cite{Houdre:2008la}. Hence we recover their
result via another approach. On the other hand, our estimate in the
right-hand-side of (\ref{eq:isop}) is sharp for small values of
$\sigma (\Lambda )$, but is worse than their estimate for large
$\sigma (\Lambda )$ since their upper bound is $8+8\sqrt{\sigma
  (\Lambda )}$.
\begin{proof}
  In order to show $h_{\mu_\sigma} \geq 1$, let us establish the
  inequality~(\ref{eq:poincare_L1}) with $h=1$. By homogeneity, it is
  sufficient to prove the result for functionals $F\in \Dom \GD$ such
  that $\esp{{\mu_\sigma}}{F} =1.$ Denote by $\nu$ the absolutely
  continuous probability measure with density $F$ with respect to the
  Poisson measure ${\mu_\sigma}$. Using duality,
  \begin{eqnarray*}
    \T _{\rho _0} ({\mu_\sigma} ,\nu ) & = & \sup _{G \in \rho _0 -\Lip _1 } \, \esp{{\mu_\sigma}}{G(F-1)} \\
    & = & \frac{1}{2} \, \sup _{ {\mu_\sigma} \, -\esssup \, \vert G\vert \leq 1} \, \esp{{\mu_\sigma}}{G(F-1)} \\
    & = & \frac{1}{2} \, \esp{{\mu_\sigma} }{\left \vert F - 1 \right \vert } .
  \end{eqnarray*}
  Hence using Theorem~\ref{T:E1} with the trivial distance $\rho _0$, we get the inequality~(\ref{eq:poincare_L1}) with $h=1$, thus obtaining the desired inequality $h_{\mu_\sigma} \geq 1$. On the other hand, to provide the upper bound in (\ref{eq:isop}),
  note that we have by the very definition of $h_{\mu_\sigma}$:
  \begin{eqnarray*}
    h_{\mu_\sigma} & \leq & \frac{2 \, {\mu_\sigma} (\partial \{ \omega (\Lambda ) =0 \} )}{{\mu_\sigma} ( \omega (\Lambda ) =0) \left( 1-{\mu_\sigma} ( \omega (\Lambda ) =0) \right)} \\
    & = & \frac{\sigma (\Lambda )}{1-e^{-\sigma (\Lambda ) }}.
  \end{eqnarray*}
The proof is achieved.
\end{proof}

\bibliographystyle{amsplain}

\end{document}